\documentclass{amsart}

\usepackage{diagbox}
\usepackage{float}
\usepackage{graphicx}
\usepackage{booktabs}
\usepackage{amsmath}
\usepackage{enumitem}
\usepackage{bm}
\usepackage{url}
\usepackage{verbatim}
\usepackage{amsmath,amsthm,amsfonts}
\usepackage{amssymb}
\usepackage{times}
\usepackage{longtable}
\usepackage{caption}
\usepackage{array}
\usepackage{tikz-cd}
\usepackage{mathtools}
\usepackage{xcolor}
\usepackage{tikz}
\usetikzlibrary{arrows,positioning}
\usepackage{bbold}
\usepackage{xcolor}
\usepackage[framemethod=TikZ]{mdframed}
\usepackage[pagebackref]{hyperref}
\hypersetup{
  colorlinks,
  linkcolor={blue!50!black},
  citecolor={green!40!black},
  urlcolor={blue!80!black}
}
\usepackage[hyphenbreaks]{breakurl}

\title{The Motivic Class of the Space of Genus 0 Maps to the Flag Variety}

\author{Jim Bryan}
\address{
Dept.\ of Math\\
University of British Columbia \\
Vancouver, British Columbia, Canada
}
\email{jbryan@math.ubc.ca}

\author{Bal\'{a}zs Elek}
\address{
  Dept.\ of Math\\
  University of New South Wales \\
  Sydney, New South Wales, Australia
}
\email{b.elek@unsw.edu.au}

\author{Freddie Manners}
\address{
  Google DeepMind / University of California San Diego \\
  California, USA
}
\email{fmanners@ucsd.edu}

\author{George Salafatinos}
\address{
  Google DeepMind \\
  Mountain View, California, USA
}
\email{salafatinos@google.com}

\author{Ravi Vakil}
\address{
Dept.\ of Math\\
Stanford University\\
Palo Alto, California, USA
}
\email{rvakil@stanford.edu}

\newtheorem{theorem}{Theorem}[section]
\newtheorem{proposition}[theorem]{Proposition}
\newtheorem{conjecture}[theorem]{Conjecture}
\newtheorem{lemma}[theorem]{Lemma}
\newtheorem{corollary}[theorem]{Corollary}

\newtheorem*{question}{Question}

\newtheorem{aitheorem}{Theorem}
\newtheorem{ailemma}[aitheorem]{Lemma}
\newtheorem{aiproposition}[aitheorem]{Proposition}

\theoremstyle{definition}

\newtheorem{remark}[theorem]{Remark}
\newtheorem{def-theorem}[theorem]{Definition-Theorem}
\newtheorem{definition}[theorem]{Definition}
\newtheorem{example}[theorem]{Example}

\newtheorem{aidefinition}[aitheorem]{Definition}

\mdfdefinestyle{promptstyle}{%
    middlelinecolor=black,
    middlelinewidth=1pt,
    backgroundcolor=black!01,
    roundcorner=0pt,
    skipabove=5pt,
    skipbelow=5pt,
    innerleftmargin=20pt,
    innerrightmargin=20pt
 }
\mdfdefinestyle{responsestyle}{%
    middlelinecolor=black,
    middlelinewidth=1pt,
    backgroundcolor=black!01,
    roundcorner=0pt,
    skipabove=5pt,
    skipbelow=5pt,
    leftmargin=20pt,
    rightmargin=20pt
 }

\newmdtheoremenv[style=promptstyle]{prompt}{Prompt}
\newmdtheoremenv[style=responsestyle]{response}{Response}

\newcommand{\CC}{{\mathbb{C}}}          
\newcommand{\RR}{{\mathbb{R}}}		
\newcommand{\ZZ}{{\mathbb{Z}}}		
\newcommand{\QQ}{{\mathbb{Q}}}		
\newcommand{\FF}{{\mathbb{F}}}

\newcommand{\kk}{{\mathbb{k}}}
\renewcommand{\AA} {{\mathbb{A}}}
\newcommand{\calE}{\mathcal{E}}

\newcommand{\calA}{\mathcal{A}}

\newcommand{\PP}{\mathbb{P}}
\newcommand{\OO}{\mathcal{O}}

\newcommand{\pt}{\operatorname{pt}}
\newcommand{\spn}{\operatorname{span}}
\newcommand{\GL}{\operatorname{GL}}
\newcommand{\PGL}{\operatorname{PGL}}

\newcommand{\Aut}{\operatorname{Aut}}
\newcommand{\Hom}{\operatorname{Hom}}

\newcommand{\Sym}{\operatorname{Sym}}

\newcommand{\Isom}{\operatorname{Isom}}
\newcommand{\Stab}{\operatorname{Stab}}

\renewcommand{\top}{\mathsf{top}}
\newcommand{\Gr}{\operatorname{Gr}}

\newcommand{\Id}{\operatorname{Id}}

\newcommand{\homotopyeq}{\simeq}

\newcommand{\Fl}{\operatorname{Fl}_{n+1}}
\newcommand{\Flnk}[1]{\operatorname{Fl}_{n+1,#1}}
\newcommand{\Var}{\mathsf{Var}}
\newcommand{\onto}{\twoheadrightarrow}
\newcommand{\LL}{\mathbb{L}}

\begin{document}

\begin{abstract}
Let $\kk$ be a field, let $\Fl$ be the variety of complete flags in
$\AA^{n+1}_{\kk}$, and let
$\Omega^{2}_{\beta}(\Fl)$ be the space of based maps $f:\PP^{1}\to
\Fl$ in the class $f_{*}[\PP^{1}]=\beta$. We show that under a mild
positivity condition on $\beta$, the class of
$\Omega^{2}_{\beta}(\Fl)$ in $K_{0}(\Var_{\kk} )$, the Grothendieck group of
varieties, is given by
\[
[\Omega^{2}_{\beta}(\Fl )] = [\GL_{n}\times \AA^{D-n^2}],
\]
where \(D=\dim\bigl(\Omega^{2}_{\beta}(\Fl)\bigr)\).

The proof of this result was obtained in conjunction with Google Gemini and related tools.
We briefly discuss this research interaction, which may be of independent interest.
However, the treatment in this paper is entirely human-authored (aside from excerpts in
an appendix which are clearly marked as such).
\end{abstract}

\maketitle 

\markboth{Motivic Class of maps to the Flag Variety}  {Bryan--Elek--Manners--Salafatinos--Vakil}

\section{Introduction}\label{sec: intro}

Throughout the paper, $\kk$ denotes an arbitrary field. By a $\kk$-variety we
mean a reduced separated scheme of finite type over $\kk$. We write
$K_0(\Var_{\kk})$ for the Grothendieck ring generated by classes $[X]$ of
$\kk$-varieties, subject to the relation $[X]=[Z]+[X\setminus Z]$ for every
closed subvariety $Z\subseteq X$, with multiplication induced by fiber
product over $\kk$.

Let $\Fl = \GL_{n+1}/B$ (where \(B\) denotes the Borel subgroup of upper triangular matrices) be the complete flag variety which parameterizes
flags of quotients:
\[
\kk^{n+1} \onto E_{n} \onto \dotsb \onto E_{1}
\]
where $\dim E_{k}=k$. There is a universal sequence of vector bundle
quotients on $\Fl$: 
\[
\OO_{\Fl}^{\oplus n+1}=\mathcal{E}_{n+1} \onto \mathcal{E}_{n}\onto \dotsb \onto
\mathcal{E}_{1}. 
\]

Let \([\Id]\in \GL_{n+1}/B\) denote the base point in \(\Fl\), and let
\(\beta \in A_1(\Fl)\) be a class in the Chow group of one-cycles modulo
rational equivalence. Let $\Omega^{2}_{\beta}(\Fl)$ be the space of genus zero
algebraic based maps in the class \(\beta\):
\smallskip
\[
\Omega^{2}_{\beta}(\Fl) = \left\{f:\PP^{1}\to \Fl\, : \quad 
f_{*}([\PP^{1}])=\beta,\quad  f([1:0])=[\Id] \right\}
\]
\smallskip

This is a quasi-projective $\kk$-scheme: it is the locally closed locus in the
Hilbert scheme of $\PP^1\times\Fl$ parameterizing graphs with the prescribed
curve class and basepoint. When taking its class in $K_0(\Var_{\kk})$, we
implicitly pass to the underlying reduced variety.

There is an isomorphism $A_{1}(\Fl )\cong \ZZ^{n}$ such that the
$n$-tuple associated to $\beta$ is given by $(d_{1},\dotsc ,d_{n})$
where
\[
d_{i} =\deg (f^{*}\mathcal{E}_{i}).
\]

We say that $\beta =(d_{1},\dotsc ,d_{n})$ is \emph{strictly monotonic} if
$0<d_{n}<d_{n-1}<\dotsb <d_{1}$. The main result of this paper is the following:

\begin{theorem}\label{thm: main result}
Suppose $\beta =(d_{1},\dotsc ,d_{n})$ is strictly monotonic. Then we have the
following equality in $K_{0}(\Var_{\kk})$:
\[
\left[\Omega^{2}_{d_{n},\dotsc ,d_{1}}(\Fl ) \right] = \left[\GL_{n}\times \AA^{D-n^{2}} \right]
\]
where $D=\sum_{k=1}^{n}2d_{k}$. 
\end{theorem}

In particular, Theorem~\ref{thm: main result} gives an equality of point counts over finite fields (see Corollary~\ref{cor:counting}).
The strict monotonicity hypothesis is a genuine restriction: for
non-strictly-monotonic multidegrees the formula in Theorem~\ref{thm: main result}
can fail; see Remark~\ref{rem:nonstrict}.

We briefly indicate the structure of the proof. We factor the passage from
complete flags to a point into a tower of spaces of maps to partial flag
varieties. The fibers at each stage are spaces of based nowhere-vanishing
sections of a vector bundle on $\PP^1$, whose motivic classes are computed in
Proposition~\ref{prop: key prop giving [N(E)]}. Although the isomorphism type of
these fibers can vary, Section~\ref{sec:motivically-trivial} stratifies the base
by splitting type and the position of the basepoint in the resulting
filtration, and proves Zariski local triviality on each stratum. Multiplying
the common fiber classes through the tower gives Theorem~\ref{thm: main result}.

Our original interest in the variety $\Omega_{\beta }^{2}(\Fl )$ is as
an algebraic counterpart to the double loop space in topology:
\[
\Omega^{2}_{\beta ,\top} (\Fl ) = \left\{f:S^{2}\to \Fl
,\,f_{*}([S^{2}])=\beta , f([1:0])=[\Id ] \right\}
\]
where we take $\Fl$ over the complex numbers with the complex analytic
topology, we consider all continuous based maps from $S^{2}\cong
\PP^{1}_{\CC}$, and $\beta \in H_{2}(\Fl )$. There is a natural
inclusion
\[
\Omega^{2}_{\beta}(\Fl )\hookrightarrow \Omega^{2}_{\beta ,\top}(\Fl )
\]
and by a theorem of Boyer, Hurtubise, Mann, and Milgram
~\cite{Boyer-Mann-Hurtubise-Milgram}
the above map is $k(\beta
)$-connected (that is, it induces isomorphisms of homotopy groups up to \(\pi_{k(\beta)}\)) where $k(\beta )\to \infty$ as $\beta \to \infty$ in a
suitable sense. In other words, the algebraic double loop space of
$\Fl$ can be viewed as a homotopy approximation to the topological
double loop space of $\Fl$ and this approximation gets better and
better as the degree increases.

By elementary homotopy theory considerations, there is a homotopy
equivalence
\[
\Omega^{2}_{\beta ,\top} (\Fl )\homotopyeq \Omega^{2}_{\top}(U(n+1))
\]
for every $\beta$ (there is no homology subscript on the right since $H_{2}(U(n+1))=0)$.

The space $\Omega^{2}_{\top}(U(n+1))$ is of fundamental interest in
topology. One of our prime interests is to explore to what extent the
algebraic double loop spaces of $\Fl$ capture the homotopy type of
$\Omega^{2}_{\top}(U(n+1))$. The \emph{rational} homotopy type of the
double loop space is known:
\[
\Omega^{2}_{\top}(U(n+1))\homotopyeq_{\QQ} U(n).
\]

So then a question of basic interest is the following:

\begin{question}
  How does the rational homotopy type of
  $\Omega^{2}_{\beta }(\Fl )$ compare to the rational homotopy type of
  $U(n)$?
\end{question}

For example, by the Boyer--Hurtubise--Mann--Milgram result, we know that
as $\beta =(d_{n},\dotsc ,d_{1})$ tends to infinity, the homotopy
groups of $\Omega^{2}_{d_{n},\dotsc ,d_{1}}(\Fl )$ and $U(n)$ agree.

A recent result of Bryan and Elek~\cite{Bryan-Elek-Postcard} says that
for the minimal strictly monotonic case ($\beta =(n,n-1,\dotsc ,1)$) there is
an isomorphism of varieties
\[
\Omega^{2}_{n,n-1,\dotsc ,1}(\Fl )\cong \GL_{n}\times \CC^{n}
\]
which has the homotopy type of $U(n)$.

Even though the homotopy type of $\Omega^{2}_{\beta}(\Fl )$ has the
``expected'' homotopy type for the minimal strictly monotonic class and for
infinitely large strictly monotonic classes, there are finite non-minimal
strictly monotonic classes for which  $\Omega^{2}_{\beta}(\Fl )$ does not have
the homotopy type of $U(n)$, not even rationally.
We give a simple case in Example~\ref{ex:homotopy-type} below.

Our main theorem can be viewed in this topological context. Our result
in the Grothendieck group implies that the weight polynomial of
$\Omega^{2}_{\beta}(\Fl )$ is the same as the weight polynomial of
$\GL_{n}$ times the affine space $\CC^{D-n^{2}}$, a variety with the homotopy
type of $U(n)$. This
suggests the possibility that the compactly supported cohomology of
$\Omega^{2}_{\beta}(\Fl )$ and $\GL_{n}\times\CC^{D-n^{2}}$ agree (for all
strictly monotonic classes $\beta )$. A slightly more optimistic conjecture is
the following:
\begin{conjecture}\label{conj: H*(Omega(Fl(n+1)),Q)=H*(U(n),Q)}
For strictly monotonic classes $\beta$, the inclusion map
\[
\Omega^{2}_{\beta}(\Fl )\hookrightarrow \Omega^{2}_{\beta ,\top}(\Fl)
\]
induces an isomorphism of rings
\[
H^{*}(\Omega^{2}_{\beta}(\Fl ),\QQ )\cong H^{*}(U(n),\QQ ).
\]
\end{conjecture}

The motivic equality in Theorem~\ref{thm: main result} controls additive
invariants such as the weight polynomial, but it does not by itself control
cup products. Thus the ring-level assertion in Conjecture~\ref{conj: H*(Omega(Fl(n+1)),Q)=H*(U(n),Q)}
is motivated by the minimal-degree and stable-range comparisons above, rather
than being a formal consequence of the motivic calculation.

\subsection{Related work}

The motivating result of Boyer--Hurtubise--Mann--Milgram
~\cite{Boyer-Mann-Hurtubise-Milgram}
 discussed above built on prior
work by Segal~\cite{Segal-1979}, Kirwan~\cite{Kirwan-1986}, Guest~\cite{Guest-1984}, Mann--Milgram~\cite{Mann-Milgram-1991,Mann-Milgram-1993} and others.
Parts of the proofs below generalize related arguments of Farb--Wolfson~\cite{Farb-Wolfson-2016} in the setting of $\Omega^2_\beta(\Gr(n+1,n))$, as well as the techniques in~\cite{Vakil-Matchett-Wood-2015}.

\subsection{The role of AI in the results of this
paper.}\label{subsec: role of AI}

We should first emphasize that the text of this paper was written from scratch by the authors: aside from the quotations in Appendix~\ref{app:output}, no part of it is AI-generated or closely rewritten AI-generated content.
In particular, the authors are responsible for the correctness of the arguments presented here in the usual way.

However, the proofs of the main results were discovered with very substantial input from Google Gemini~\cite{gemini} and related tools: specifically, \emph{DeepThink}~\cite{deepthink}; and a related unpublished system specialized for mathematics developed by the fourth author, also built on Gemini, and provisionally named \emph{FullProof}.

We now briefly describe the detailed nature of this collaboration.
Further mathematical details are given in Appendix~\ref{app:ai-details}.

In the first instance, the first, second and fifth authors isolated, as a conjecture, a somewhat weaker version of Theorem~\ref{thm: main result} (conjecturing the number of points of the variety $\Omega_\beta^2(\Fl)$ over a finite field $\FF_q$; see Corollary~\ref{cor:counting} below).
In previous work, they had used conventional computational tools (in particular, \emph{Macaulay2}~\cite{M2}) to obtain strong numerical evidence for this conjecture.

This weaker question was then framed as a scaffolding of sub-problems of increasing difficulty, finishing with the problem itself: for example, by first asking for a proof of the result for small values of $n$ and $\beta$.
This scaffolding of sub-problems was then either prompted to the AI system in one go; or in turn, with each new subproblem being annotated with a successful AI response to one or more of the previous subproblems.
We give some representative examples in Appendix~\ref{app:prompts}.

In the experiments reported here, none of the systems used successfully proved Corollary~\ref{cor:counting} without assistance using the original scaffolding. Instead, the proof was obtained by an iterative human/AI interaction, of the following flavor.
\begin{itemize}
  \item The AI systems provided correct and readable solutions to some earlier problems in the scaffolding, corresponding to special cases.  However, even given these solutions as context, they could not generalize them to the full problem.
  \item A close human analysis of the AI output isolated key intermediate statements which the human mathematicians could see how to generalize, thereby suggesting a proof strategy for the general case.
  \item The AI systems were re-prompted with new questions which hinted at this proof strategy, either with or without the old successful solutions in context.
  \item The hinted approach was enough for the system to generate complete proofs to the new problems, and thereby solve the original (slightly weaker) conjecture.
  \item Finally, the AI systems were asked to extend the arguments to prove the full statement, Theorem~\ref{thm: main result}.
\end{itemize}

Since the authors have opted to present a purely human account of these results in this paper, in some cases the proofs below bear only a high-level resemblance to those suggested by AI tools.  However, it is worth noting that some of the AI-generated proofs -- and in particular those derived from the specialized internal tool \emph{FullProof} -- are already very accomplished.
We give some examples of this model output in Appendix~\ref{app:output}. 

It is natural to ask how close the resemblance is between the AI-contributed proofs, and prior literature that Gemini is likely to have seen in its training data.\footnote{As run, none of these systems had access to the internet or other search tools.}
Certainly the latter includes related work such as
~\cite{Boyer-Mann-Hurtubise-Milgram,Farb-Wolfson-2016},
and it seems likely that being able to build on these arguments made the problem more tractable for the AI systems than some other research problems.
However, the model outputs (such as the one in Appendix~\ref{app:output}) do not appear to the authors to be \emph{that} close to those or (to the best of our knowledge) any other pre-existing sources.
So, absent some future discovery to the contrary, the model's contribution appears to involve a genuine combination of synthesis, retrieval, generalization and innovation of these existing techniques.

\subsection{Acknowledgements}

First, the authors wish to thank Google DeepMind for their extensive support and
collaboration.
In particular, they are grateful to Adam Brown, Vinay Ramasesh and others on the Blueshift team
for supporting, advising and enabling this project.
FM and GS are moreover grateful to Google DeepMind for providing employment and hospitality
for the duration of this work.

FM is also supported by a Sloan Fellowship.

\section{Proof of the main result}

\subsection{Partial flags and a tower of fibrations}%
\label{subsec:partial-flags}

For $1 \le k \le n+1$, we let $\Flnk{k}$ denote the moduli space of partial flag quotients:
\[
\Flnk{k} = \left\{\kk^{n+1} \onto V_{n}\onto \dotsb \onto V_{k} \right\}
\]
where $\dim V_{a} = a$.
There are maps $\Flnk{k} \to \Flnk{k+1}$ for $1 \le k \le n$ by omitting $E_k$.
We have $\Flnk{1}=\Fl$ and $\Flnk{n+1} = \pt$.

There is a universal partial flag quotient of vector bundles on $\Flnk{k}$,%
\footnote{Here we abuse notation slightly by writing $\calE_i$ for the
universal vector bundles with base $\Flnk{k}$ for every $k$.  The
$\calE_i$ for different $k$ are compatible under pullback by $\Flnk{k}
\to \Flnk{k+1}$ in the obvious way.} 
\[
  \OO_{\Flnk{k}}^{n+1} \onto \calE_n \onto \dotsb \onto \calE_k.
\]

Exactly as above, we write $\Omega^2_{d_n,\dotsc,d_k}(\Flnk{k})$ for the space of maps
\[
  \Omega^2_{d_n,\dotsc,d_k}(\Flnk{k}) = \bigl\{ f \colon \PP^1 \to \Flnk{k} : \deg(f^\ast(\calE_i)) = d_i,\ f([1:0]) = \phi_k \bigr\}
\]
where $\phi_k \in \Flnk{k}$ is the standard partial flag
\[
  \kk^{n+1} \onto \kk^{n+1} / \spn(e_{n+1}) \onto \dots \onto \kk^{n+1} / \spn(e_{k+1},\dots,e_{n+1});
\]
i.e., the image of the standard flag $[\Id] \in \Fl$ under the map $\Fl \to \Flnk{k}$.

A closed point $f \in \Omega^{2}_{d_{n},\dotsc ,d_{k}}(\Flnk{k})$
corresponds to a sequence of vector bundle quotients on $\PP^{1}$:
\[
\{\OO^{\oplus (n+1)}_{\PP^{1}}\onto E_{n}\onto \dotsb \onto E_{k} \}
\]
where $E_{j}=f^{*}(\mathcal{E}_{j})$. We write
\begin{equation}\label{eq:pik-def}
  \pi_{k}: \Omega^2_{d_{n},\dotsc ,d_{k}}(\Flnk{k}) \to
  \Omega^2_{d_{n},\dotsc ,d_{k+1}}(\Flnk{k+1})  
\end{equation}
for the map induced by $\Flnk{k} \to \Flnk{k+1}$.

\begin{definition}\label{defn: N(F)}
  Let $\phi \colon F \to \PP^1$ be a vector bundle on $\PP^1$ of rank $r$,
  and let $p$ (the basepoint) be any projectivized vector over $[1:0]$, i.e., $p \in \PP(\phi^{-1}([1:0])) \cong \PP^{r-1}$.  
  We write:
  \begin{itemize}
    \item $\Sigma(F) = \PP (H^0(F))$, the space of nonzero sections of $F$ up to scalars $\kk^\ast$;
    \item $N(F) \subseteq \Sigma(F)$ for the (possibly empty) space of \emph{nowhere vanishing} sections of $F$, up to scalars;
    \item $\Sigma_p(F)$ for the \emph{based} nonzero sections:
      \[
        \Sigma_p(F) = \bigl\{ \sigma = [\sigma_1:\dotsc:\sigma_r] \in \Sigma(F) \colon [\sigma_1([1:0]):\dotsc:\sigma_r([1:0])] = p \bigr\}
      \]
      (so in particular, $(\sigma_1([1:0]),\dots,\sigma_r([1:0])) \ne 0$);
    \item $N_p(F)$ for the based nowhere vanishing sections, $N_p(F) = \Sigma_p(F) \cap N(F)$;
    \item \(F(d)\) for the twist \(F\otimes \OO(d)\).
  \end{itemize}
\end{definition}

\begin{lemma}\label{lem: fibers of Map(Fl(n+1,k))-->Map(Fl(n+1,k+1))}
  Suppose that $d_{n}\leq d_{n-1}\leq \dotsb \leq d_{1}$ and $1 \le k \le n$.
  For any $f \in \Omega^2_{d_n,\dotsc,d_{k+1}}(\Flnk{k+1})$, the fiber $\pi_k^{-1}(f) \subseteq \Omega^2_{d_n,\dots,d_k}(\Flnk{k})$ obeys
  \[
    \pi_k^{-1}(f)  \cong N_{v_{k+1}}(E_{k+1}(d_k-d_{k+1})),
  \]
  where by convention $d_{n+1}=0$, and $v_{k+1}$ is the image of the standard basis vector $e_{k+1} \in \kk^{n+1}$ in the standard quotient $\kk^{n+1} / \spn(e_{k+2},\dots,e_{n+1})$.
\end{lemma}
\begin{proof}
  By the universal property, $f$ corresponds to a partial flag
  \[
    \OO^{n+1}_{\PP^1} \onto E_n \onto E_{n-1} \onto \dots \onto E_{k+1}
  \]
  where $E_i = f^{\ast}(\calE_i)$, and any $g \in \pi_k^{-1}(f)$ corresponds to a partial flag
  \[
    \OO^{n+1}_{\PP^1} \onto E_n \onto E_{n-1} \onto \dots \onto E_{k+1} \onto E_k
  \]
  which refines $f$ (i.e., which agrees with $f$ up to $\onto E_{k+1}$).

  Hence, given $f$, specifying a choice of $g \in \pi_k^{-1}(f)$ is equivalent to specifying a vector bundle quotient $E_{k+1} \onto E_k$ such that
  \begin{enumerate}[label=(\roman*)]
    \item $E_k$ has rank $k$;
    \item the fiber $E_k|_{[1:0]}$ is the standard quotient $\kk^{n+1} / \spn(e_{k+1},\dots,e_{n+1})$; and
    \item $\deg(E_k) = d_k$.
  \end{enumerate}
  Specifying such a quotient is equivalent to specifying its kernel $R = \ker(E_{k+1} \onto E_k)$; i.e., to a choice of rank $1$ subbundle $R \subseteq E_{k+1}$.
  The equivalent requirements on $R$ are that $R|_{[1:0]} = \spn(v_{k+1})$ (where, as in the statement, $v_{k+1}$ is the image of the standard basis vector $e_{k+1}$ in the standard quotient), and $\deg(R) = \deg(E_{k+1})-\deg(E_k) = d_{k+1}-d_k$.

Such subbundles $R$ correspond to vector bundle embeddings
$\OO_{\PP^1} (d_{k+1}-d_k) \to E_{k+1}$, up to multiplication by
scalars $\kk^\times \cong \Aut(\OO_{\PP^1} (d_{k+1}-d_k))$.  These are
further equivalent to embeddings $\OO_{\PP^1} \to
E_{k+1}(d_k-d_{k+1})$ (up to scalars); that is, to nowhere vanishing
sections of $E_{k+1}(d_k-d_{k+1})$ (up to scalars).  Finally, the
condition over $[1:0]$ exactly corresponds to the statement that this
section lies in $N_{v_{k+1}}(E_{k+1}(d_k-d_{k+1}))$.
\end{proof}

In the following, we denote by \(\LL=\left[\AA^1_{\kk}\right]\) the class of the affine line in \(K_0(\Var_{\kk})\).

\begin{proposition}\label{prop: key prop giving [N(E)]}
  Suppose that $\phi \colon E \to \PP^1$ is a vector bundle of rank $r$ and degree $d$ on
  $\PP^{1}$, and suppose that $H^{1}(E(-2))=0$; equivalently, that $E
  \cong \bigoplus_{i=1}^r \OO_{\PP^{1}}(a_i)$ where $a_i \ge 1$. 
  
  Then for any basepoint $p \in \PP(\phi^{-1}([1:0])) \cong \PP^{r-1}$,
  the class of $N_p(E)$ in $K_{0}(\Var_{\kk})$ is given by
  \[
    [N_p(E)] = \LL^{d-r+1} (\LL^{r-1}-1).
  \]
  In particular, $[N_p(E)]$ is independent of the splitting type of $E$.
\end{proposition}
The special case where $E = \OO_{\PP^{1}}(a)^r$ for $a \ge 1$ (i.e.,
all rank $1$ summands of $E$ have the same degree), $p=[1:\dots:1]$,
and $\kk$ has characteristic $0$, is implied
by~\cite[Theorem~1.2]{Farb-Wolfson-2016}, and there are many
similarities between the proofs.
\begin{proof}
  For any rank $r$ bundle $F\cong\oplus_{i=1}^{r}\OO_{\PP^1} (a_{i})$ with $a_i \ge 0$,
  note $H^0(F)$ may be identified with tuples $(s_1,\dots,s_r)$ of homogeneous polynomials $s_i(x_0,x_1)$
  of degree $a_i$.
  It follows that $\Sigma(F) \cong \PP^{\deg(F) + r-1}$.
  Moreover, $s = [s_1 : \dotsc: s_r] \in \Sigma(F)$ lies in $\Sigma_p(F)$ if and only if the leading coefficients $s_i(1,0)$ satisfy
  $(s_1(1,0),\dotsc,s_r(1,0))\neq 0$ and $[s_1(1,0):\dotsc:s_r(1,0)] = p$.
  Hence $\Sigma_p$ lies inside $\PP^{\deg(F) + r - 1}$ as the intersection of $r-1$ independent hyperplanes and one independent hyperplane complement; therefore,
  \begin{equation}
    \label{eq:sigma-class}
    \Sigma_p(F) \cong \AA^{\deg(F)}.
  \end{equation}

  For any $k \ge 0$, let
  $V_{k}(F)\subset \Sigma(F)$ be the (possibly empty) subset
  \[
    V_{k}(F) = \bigl\{[s_{1}:\dotsb :s_{r}]: \text{ $\gcd
    (s_{1},\dotsc ,s_{r})$ has degree $k$} \bigr\}.
  \]
  Recalling that $N(F)\subset \PP (H^{0}(F))$ is the (unbased) locus of nowhere
  vanishing sections, we have an isomorphism
  \[
    V_{k}(F) \cong N(F(-k))\times \PP (H^{0}(\OO_{\PP^1} (k)))
  \]
  given by 
  \[
    [s_{1}:\dotsb :s_{r}]\mapsto \left(\left[\frac{s_{1}}{g}:\dotsb
    :\frac{s_{r}}{g} \right],[g] \right)
  \]
  where $g=\gcd (s_{1},\dotsc ,s_{r})$ is well defined up to
  multiplication by a scalar.

  For $\sigma = [s_1:\dotsb:s_r]$ and $g$ as above,
  notice that $\sigma([1:0]) = p$ if and only if 
  \[
    g(1,0) \ne 0 \qquad \text{and} \qquad \left[ \frac{s_1}{g}(1,0) : \dotsb : \frac{s_r}{g}(1,0) \right] = p.
  \]
  Hence, the above isomorphism restricts to an isomorphism
  \[
    V_k(F) \cap \Sigma_p \cong N_p(F(-k)) \times \AA^{k}.
  \]

  Moreover, $\Sigma(F)$ (respectively, $\Sigma_p(F)$) admits a locally closed stratification
  \begin{align*}
    \Sigma(F) &= \bigcup_{k\geq 0} V_{k}(F) = \bigcup_{k \geq 0} N(F(-k)) \times \PP^{k}, \\
    \Sigma_p(F) &= \bigcup_{k \ge 0} V_k(F) \cap \Sigma_p(F) 
    = \bigcup_{k \ge 0} N_p(F(-k)) \times \AA^{k}.
  \end{align*}
  Taking classes in $K_0(\Var_\kk)$,
  \[
    [\Sigma_p(F)] = \sum_{k \ge 0} [N_p(F(-k))] \LL^k.
  \]
  Applying this with $F=E$ and $F=E(-1)$ gives
  \[
    [\Sigma_p(E)] - \LL [\Sigma_p(E(-1))] = [N_p(E)].
  \]
  Note that, by hypothesis, $E$ and $E(-1)$ have all their splitting summands of non-negative degree, so
  by~\eqref{eq:sigma-class} we get
  \[
    [\Sigma_p(E)] = \LL^d,\qquad [\Sigma_p(E(-1))] = \LL^{d-r}
  \]
  and so
  \[
    [N_p(E)] = \LL^d - \LL^{d-r+1} = \LL^{d-r+1} (\LL^{r-1} - 1)
  \]
  as claimed.
\end{proof}

\begin{remark}\label{rem:nonstrict}
  If $F$ has a summand of degree $<0$, the space $\Sigma_p(F)$ is either $\AA^r$ for some explicit $r$ or empty, depending on whether the corresponding coordinate of $p$ is zero or non-zero.
  Hence, if $E$ is only required to have all splitting degrees $a_i \ge 0$, the classes $[\Sigma_p(E(-1))]$ and $[N_p(E)]$ can be computed as above, 
  but depend non-trivially on the interaction of $p$ with the splitting of $E$.
  Consequently, the fiber computation used in the proof of
  Theorem~\ref{thm: main result} no longer has a uniform answer when adjacent
  degrees coincide. This reflects a genuine restriction on the theorem:
  examples with non-strictly-monotonic multidegree show that the formula in
  Theorem~\ref{thm: main result} can fail.
\end{remark}

\begin{corollary}%
  \label{cor:fiber-type}
  If $d_{n}<d_{n-1}<\dotsb <d_{1}$, then for all $f \in \Omega^2_{d_n,\dots,d_{k+1}}(\Flnk{k+1})$, the fibers $\pi_{k}^{-1}(f)$ have class $[\pi_k^{-1}(f)] = \LL^{(k+1) d_k - k d_{k+1} - k} (\LL^k - 1)$ in $K_0(\Var_\kk)$, independent of $f$.
\end{corollary}
\begin{proof}
  Recall $f \in \Omega^2_{d_n,\dotsc,d_{k+1}}(\Flnk{k+1})$ corresponds to a flag
  \[
    \OO^{n+1}_{\PP^1} \onto E_n \onto E_{n-1} \onto \dots \onto E_{k+1}
  \]
  where $E_{i} = f^\ast(\calE_{i})$.
  Since $E_{k+1}$ is a quotient of $\OO_{\PP^1}^{\oplus (n+1)}$, it has no
  negative summands in its splitting type.
  Then by the strict monotonicity
  assumption, $d_{k}-d_{k+1}>0$ and so $E_{k+1}(d_{k}-d_{k+1})$ has all
  strictly positive summands.
  Moreover, $E_{k+1}(d_k - d_{k+1})$ has rank $k+1$ and degree $d_{k+1} + (k+1) (d_k - d_{k+1})$.
  Applying Lemma~\ref{lem: fibers of Map(Fl(n+1,k))-->Map(Fl(n+1,k+1))} and Proposition~\ref{prop: key prop giving [N(E)]} gives
  \[
    [\pi_k^{-1}(f)] = \LL^{(k+1) d_k - k d_{k+1} - k} (\LL^k - 1)
  \]
  as required.
\end{proof}

If $\pi_k$ were a Zariski locally trivial fibration, we could immediately relate 
the classes $[\Omega^2_{d_n,\dots,d_{k+1}}(\Flnk{k+1})]$ and $[\Omega^2_{d_n,\dots,d_k}(\Flnk{k})]$.
However it is not: the fibers need not even be isomorphic as varieties, as shown in Example~\ref{ex:fibers-different} below.

Given this, we introduce the following terminology.

\begin{definition}\label{defn: motivically trivial fibration}
  We say that a map $\pi :Y\to X$ is a \emph{motivically trivial fibration} if
there is a locally closed stratification of $X=\cup_{i} Z_{i}$ such that
$\pi^{-1}(Z_{i})\to Z_{i}$ is a Zariski locally trivial fiber bundle
with fibers $F_{i}$ such that $[F_{i}]=[F]\in K_{0}(\Var_{\kk})$ is
independent of $i$. We note that the equation $[Y]=[X]\cdot [F]$ then
holds in $K_{0}(\Var_{\kk})$. 
\end{definition}

\begin{proposition}\label{prop: Map(Fl(n+1,k))-->Map(Fl(n+1,k+1)) is motivically trivial}
  For $d_{n}<d_{n-1}<\dotsb <d_{1}$ and $1 \le k \le n$,
  the map $\pi_k$ is a motivically trivial fibration with fiber class $\LL^{(k+1) d_k - k d_{k+1} - k} (\LL^k - 1)$.
\end{proposition}

It remains to show that $\Omega^2_{d_n,\dots,d_{k+1}}(\Flnk{k+1})$ admits a locally closed stratification such that $\pi_k$ is Zariski locally trivial on strata.
We defer the proof of this to Section~\ref{sec:motivically-trivial}.

\subsection{Finishing the proof of the Main Theorem}

Given Proposition~\ref{prop: Map(Fl(n+1,k))-->Map(Fl(n+1,k+1)) is motivically trivial}, and recalling $\Fl = \Flnk{1}$ and $\Flnk{n+1} = \pt$, by induction on $k$ we have
\begin{align}\label{eqn: first formula for [Map(P1,Fl)]}
\left[\Omega^2_{d_{n},\dotsc ,d_{1}}(\Fl ) \right] &=
\prod_{k=1}^{n} \LL^{(k+1)d_{k}-kd_{k+1}-k} \cdot (\LL^{k}-1) \nonumber\\
&=\LL^{\alpha}\cdot \prod_{k=1}^{n}(\LL^{k}-1)
\end{align}
where 
\begin{align*}
  \alpha &= \sum_{k=1}^{n} \bigl((k+1)d_{k}-kd_{k+1}-k \bigr)\\
&=\sum_{k=1}^{n}(k+1)d_{k} - \sum_{k=2}^{n+1}(k-1)d_{k}
-\binom{n+1}{2} \\
&=\sum_{k=1}^{n}(k+1)d_{k} - \sum_{k=1}^{n}(k-1)d_{k}
-\binom{n+1}{2} \\
&= \sum_{k=1}^{n}2d_{k} -\binom{n+1}{2}\\
&=D-\binom{n+1}{2}.
\end{align*}
Note that we've used the convention $d_{n+1}=0$ to change the summation
limits from the second to the third line above.

We now note that 
\[
\left[\GL_{n} \right] =\LL^{\binom{n}{2}}\prod_{k=1}^{n}(\LL^{k}-1).
\]

We thus now see that equation~\eqref{eqn: first formula for
[Map(P1,Fl)]} becomes
\begin{align}\label{eqn: formula for [Map(P1,Fl)]}
[\Omega^2_{d_{n},\dotsc ,d_{1}}(\Fl )] &= \LL^{D-\binom{n+1}{2}-\binom{n}{2}}\,  [\GL_{n}]\nonumber \\
&= \LL^{D-n^{2}}\, [\GL_{n}] \nonumber
\end{align}
as required.

\subsection{Motivically trivial fibrations}%
\label{sec:motivically-trivial}

We now prove Proposition~\ref{prop: Map(Fl(n+1,k))-->Map(Fl(n+1,k+1)) is motivically trivial}.

Let
\[
  B_k=\Omega^2_{d_n,\dots,d_{k+1}}(\Flnk{k+1})
\]
and let $p\colon\PP^1\times B_k\to B_k$ be the projection. The universal
map over $B_k$ determines a universal partial quotient flag, and in
particular a rank $k+1$ vector bundle $\calE_{k+1}$ on
$\PP^1\times B_k$. The basepoint condition gives a nowhere-zero section
$v_{k+1}$ of $\calE_{k+1}|_{\{[1:0]\}\times B_k}$; this is the family
version of the vector defined in
Lemma~\ref{lem: fibers of Map(Fl(n+1,k))-->Map(Fl(n+1,k+1))}.

We recall that an algebraic group $G$ is called \emph{special} if every
principal $G$-bundle over a variety is Zariski locally trivial. The groups
$\GL_m$ and $\mathbb G_a$ are special, and an extension of special groups is
special; see~\cite{Serre-Special-Groups,Reichstein-Tossici-Special-Groups}.

\newcommand{\vecr}[0]{\mathbf{r}}
Given a finitely supported sequence $\vecr=(r_0,r_1,r_2,\dots)$ of
non-negative integers, write
\[
  F_{\vecr}=\bigoplus_{j\geq 0}\OO_{\PP^1}(j)^{r_j}
\]
and let $G_{\vecr} = \Aut(F_{\vecr})$.
We briefly argue that $G_{\vecr}$ is special.\footnote{A previous version of this paper had an argument which required $\kk$ to be a perfect field; we thank the anonymous referee for pointing this out.}

The action of $G_{\vecr}$ on $F_{\vecr}$ is ``block lower-triangular'' in the following sense.
There is a natural filtration $F_{\vecr} = F_{\vecr}^{\geq 0} \supseteq F_{\vecr}^{\geq 1} \supseteq \dots$ given by
\[
  F_{\vecr}^{\geq \ell} = \bigoplus_{j \geq \ell} \OO_{\PP^1}(j)^{r_j}.
\]
Since $\Hom(\OO_{\PP^1}(j), \OO_{\PP^1}(j')) = 0$ if $j'<j$, any $\gamma \in G_{\vecr}$
obeys $\gamma(F_{\vecr}^{\geq i}) \subseteq F_{\vecr}^{\geq i}$.
For $k \geq 0$, define
\[
  G_{\vecr}^{(k)} := \bigl\{ \gamma \in G_{\vecr} \colon (\gamma - \Id)(F_{\vecr}^{\geq \ell}) \subseteq F_{\vecr}^{\geq (\ell+k)} \ \forall \ell \geq 0 \bigr\} \unlhd G_{\vecr}.
\]
We have $G_{\vecr}^{(0)} = G_{\vecr}$, and if $k > \max \{ j \colon r_j > 0 \}$ then $G_{\vecr}^{(k)}$ is trivial.
Furthermore, it is straightforward to check that $G_{\vecr}^{(0)} / G_{\vecr}^{(1)} \cong \prod_{j \geq 0} \GL_{r_j}$
and $G_{\vecr}^{(k)} / G_{\vecr}^{(k+1)}$ is a power of $\mathbb G_a$ for every $k \geq 1$.
By repeated application of the properties of special groups listed above,
it follows that $G_{\vecr}$ is special.

The degree decomposition of $F_{\vecr}$ gives a decomposition of the fiber $F_{\vecr}|_{[1:0]}$:
\[
  F_{\vecr}|_{[1:0]}=\bigoplus_{j\geq 0}\kk^{r_j}.
\]
For a non-zero vector $u$ in this fiber, define its \emph{depth} $\delta(u)$
to be the largest integer $\ell\geq 0$ such that
$u\in\bigoplus_{j\geq\ell}\kk^{r_j}$. If $r_\ell>0$, let $u_\ell$ be the
vector whose component in $\kk^{r_\ell}$ is $(1,\dots,1)$ and whose other
components vanish. By the block lower-triangular structure of $\Aut(F_{\vecr})$ discussed above,
its action on the fiber $F_{\vecr}|_{[1:0]}$ preserves the depth filtration and is
transitive on the vectors of each fixed depth. Write
$G_{\vecr,\ell}\leq G_{\vecr}$ for the stabilizer of $u_\ell$.
The group $G_{\vecr,\ell}$ is an extension of
\[
  \left(\prod_{j \ne \ell} \GL_{r_j} \right) \times \left(\GL_{r_\ell-1} \ltimes \mathbb G_a^{r_\ell-1}\right)
\]
by $G_{\vecr,\ell} \cap G_{\vecr}^{(1)}$, which is again split unipotent by considering the filtration on $G_{\vecr}^{(1)}$, and hence $G_{\vecr,\ell}$ is special.

The bundle $\calE_{k+1}$ is a quotient of the trivial bundle, so its
splitting summands on every geometric fiber have non-negative degree. For
fixed $d_{k+1}$, let $I$ be the finite set of sequences $\vecr$ satisfying
\[
  \sum_{j\geq0}r_j=k+1,
  \qquad
  \sum_{j\geq0}j r_j=d_{k+1}.
\]
By the Birkhoff--Grothendieck theorem~\cite{MR662762}, every geometric fiber
of $\calE_{k+1}$ is isomorphic to one of the bundles $F_{\vecr}$ with
$\vecr\in I$. The relative Harder--Narasimhan stratification~\cite{Nitsure-HN}
gives the locus $B_{k,\vecr}\subseteq B_k$ on which the
splitting type is $\vecr$ the structure of a locally closed subvariety. On
$\PP^1$, the Harder--Narasimhan type is equivalent to the splitting type, and
over $B_{k,\vecr}$ the universal bundle has a relative filtration
$\calE_{\geq j}$ with graded pieces of the form
\[
  \operatorname{gr}_j(\calE_{k+1})
  \cong \OO_{\PP^1}(j)\boxtimes V_j,
\]
where $V_j$ is a vector bundle of rank $r_j$ on $B_{k,\vecr}$.

For $\vecr\in I$ and $r_\ell>0$, define
\[
  S_{\vecr,\ell}
  =\Omega^2_{d_n,\dots,d_{k+1};\vecr,\ell}(\Flnk{k+1})
  \subseteq B_{k,\vecr}
\]
by requiring that, over $[1:0]$, the section $v_{k+1}$ lie in
$\calE_{\geq\ell}$ and have non-zero image in
$\operatorname{gr}_\ell(\calE_{k+1})$.
The first condition is closed and the second is open within it, so
$S_{\vecr,\ell}$ is locally closed. Equivalently, it is the locus where
$E_{k+1}\cong F_{\vecr}$ and the distinguished vector has depth $\ell$.
Consequently, the varieties $S_{\vecr,\ell}$ form a locally closed
stratification of $B_k$.

We next record the local framing statement needed below.

\begin{lemma}\label{lem:local-pointed-framing}
  Over a Zariski open cover $\{U_\alpha\}$ of $S_{\vecr,\ell}$, there are
  isomorphisms
  \[
    \psi_\alpha\colon
    \calE_{k+1}|_{\PP^1\times U_\alpha}
    \xrightarrow{\cong}
    F_{\vecr}\boxtimes\OO_{U_\alpha}
  \]
  such that $\psi_\alpha(v_{k+1})=u_\ell$ over
  $\{[1:0]\}\times U_\alpha$.
\end{lemma}

\begin{proof}
  More explicitly, if $p:\PP^1\times S_{\vecr,\ell}\to S_{\vecr,\ell}$,
  then
  $V_j=p_*(\operatorname{gr}_j(\calE_{k+1})\otimes\OO_{\PP^1}(-j))$.
  Cohomology and base change shows that $V_j$ is locally free of rank $r_j$
  and that the evaluation map
  $\OO_{\PP^1}(j)\boxtimes V_j\to\operatorname{gr}_j(\calE_{k+1})$
  is an isomorphism. Zariski locally on $S_{\vecr,\ell}$, the
  bundles $V_j$ are trivial. After shrinking further to affine open subsets,
  the extensions in the relative Harder--Narasimhan filtration split. Indeed,
  the obstruction groups between a degree-$j$ graded piece and the
  higher-degree pieces vanish relatively because
  \[
    H^1(\PP^1,\OO_{\PP^1}(i-j))=0\qquad\text{for }i>j.
  \]
  We therefore obtain local isomorphisms
  $\theta_\alpha\colon\calE_{k+1}\xrightarrow{\cong}F_{\vecr}$.

  Under such an isomorphism, $v_{k+1}$ determines a morphism from
  $U_\alpha$ to the orbit of depth-$\ell$ vectors in
  $F_{\vecr}|_{[1:0]}$. The orbit map from $\Aut(F_{\vecr})$ to this locus is
  a principal $G_{\vecr,\ell}$-bundle. Since $G_{\vecr,\ell}$ is special,
  its pullback to $U_\alpha$ is Zariski locally trivial. Refining the cover
  and composing $\theta_\alpha$ with the resulting automorphisms of
  $F_{\vecr}$ gives the required isomorphisms $\psi_\alpha$.
\end{proof}

It remains to prove local triviality of $\pi_k$ on these strata.

\begin{lemma}\label{lem:zlt}
  For every $\vecr\in I$ and every $\ell\geq0$ with $r_\ell>0$, the base
  change
  \[
    \pi_k^{-1}(S_{\vecr,\ell})\longrightarrow S_{\vecr,\ell}
  \]
  is Zariski locally trivial, with fiber
  $N_{u_\ell}(F_{\vecr}(d_k-d_{k+1}))$.
\end{lemma}

\begin{proof}
  Let $\calA_{d_n,\dots,d_{k+1};\vecr,\ell}$ be the relative isomorphism
  space
  \[
    \Isom_{\PP^1\times S_{\vecr,\ell}/S_{\vecr,\ell}}
    \bigl((\calE_{k+1},v_{k+1}),
    (F_{\vecr}\boxtimes\OO_{S_{\vecr,\ell}},u_\ell)\bigr).
  \]
  Thus it parameterizes pairs $(f,\psi)$ consisting of a partial quotient
  flag $f$ in $S_{\vecr,\ell}$ and an isomorphism
  $\psi\colon E_{k+1}\xrightarrow{\cong}F_{\vecr}$ taking $v_{k+1}$ to
  $u_\ell$. Let
  \[
    \tau_{\vecr,\ell}\colon
    \calA_{d_n,\dots,d_{k+1};\vecr,\ell}
    \longrightarrow S_{\vecr,\ell}
  \]
  forget $\psi$.

  The group $G_{\vecr,\ell}$ acts by postcomposition. Functorially over any
  test scheme, two pointed framings $\psi$ and $\psi'$ differ by the unique
  element $\psi'\psi^{-1}$ of $G_{\vecr,\ell}$. Hence the canonical map
  \[
    \calA_{d_n,\dots,d_{k+1};\vecr,\ell}\times G_{\vecr,\ell}
    \longrightarrow
    \calA_{d_n,\dots,d_{k+1};\vecr,\ell}
    \times_{S_{\vecr,\ell}}
    \calA_{d_n,\dots,d_{k+1};\vecr,\ell}
  \]
  is an isomorphism. Lemma~\ref{lem:local-pointed-framing} supplies Zariski
  local sections of $\tau_{\vecr,\ell}$, and therefore identifies its
  restriction over each $U_\alpha$ with
  $U_\alpha\times G_{\vecr,\ell}$. This proves directly that
  $\tau_{\vecr,\ell}$ is a Zariski locally trivial principal
  $G_{\vecr,\ell}$-bundle; in particular, it is flat and locally finitely
  presented.

  Choose a section of $\tau_{\vecr,\ell}$ over $U_\alpha$, equivalently a
  family of pointed framings $\psi_\alpha$. Every element of
  $N_{u_\ell}(F_{\vecr}(d_k-d_{k+1}))$ corresponds, as in
  Lemma~\ref{lem: fibers of Map(Fl(n+1,k))-->Map(Fl(n+1,k+1))}, to a quotient
  $F_{\vecr}\onto E_k$ with $\deg(E_k)=d_k$ and with the prescribed quotient
  over $[1:0]$. Composing this quotient with $\psi_\alpha$ gives an
  isomorphism
  \[
    U_\alpha\times N_{u_\ell}(F_{\vecr}(d_k-d_{k+1}))
    \xrightarrow{\cong}\pi_k^{-1}(U_\alpha).
  \]
  Its inverse recovers the nowhere-vanishing section from the kernel of
  $E_{k+1}\onto E_k$. Thus these are local trivializations of $\pi_k$ on
  $S_{\vecr,\ell}$.
\end{proof}

\subsection{Negative examples}%
\label{subsec:negative-examples}

We briefly outline some examples where key equalities above in terms of motivic class cannot be upgraded to analogous statements for isomorphism class or homotopy type.

We first show that the variety $\Omega^2_{d_n,\dotsc,d_1}(\Fl)$ itself need not have the rational homotopy type of $\GL_n$ (over $\CC$), even for strictly monotonic classes $(d_n,\dotsc,d_1)$.  As discussed in Section~\ref{sec: intro}, this means a (rational) homotopy version of Theorem~\ref{thm: main result} does not hold.

\begin{example}%
  \label{ex:homotopy-type}
  Consider $\Omega^2_{d_n,\dots,d_1}(\Fl)$ for $\kk=\CC$ and $n=1$.
  Then $\Fl =\PP^{1}$ and $\Omega^{2}_{d}(\PP^{1})$ is the space
  of degree $d$ maps $f:\PP^{1}\to \PP^{1}$ obeying the basepoint condition $f([1:0])=[1:0]$.
  A general degree $2$ map $f \colon \PP^1 \to \PP^1$ is
  uniquely determined, up to the (free) action of $\PGL_2$ on the codomain, by its unordered pair of (distinct) critical points in the domain.
  Taking into account the basepoint condition,
  this gives $\Omega^2_2(\PP^1)$ the structure of an $H$-principal bundle over $\Sym^2 \PP^1 \setminus \Delta$,
  where $H = \Stab_{\PGL_2}([1:0]) \cong \CC^\ast \ltimes \CC$.

  Up to homotopy, $\Sym^2 \PP^1 \setminus \Delta \simeq \RR \PP^2$ (by deforming unordered pairs of distinct points in $S^2$ to unordered pairs of antipodal points).
  Hence $\Omega^2_2(\PP^1)$ has the homotopy type of a principal $S^{1}$ bundle over $\RR \PP^{2}$.
  There are only two such bundles (as $H^2(\RR \PP^2, \ZZ) \cong \ZZ/2\ZZ)$, namely the trivial bundle $S^1 \times \RR \PP^2$ and the mapping torus of the antipodal map $S^2 \to S^2$.  Both are double-covered by $S^1 \times S^2$, so $\pi_2(\Omega^2_2(\PP^1)) \cong \ZZ$, and in particular $\Omega^2_2(\PP^1)$ does not have the rational homotopy type of $U(1) \simeq \GL_1$.
\end{example}

By a related argument, we can show that the projection $\pi_k \colon \Omega^2_{d_n,\dots,d_{k}}(\Flnk{k}) \to \Omega^2_{d_n,\dots,d_{k+1}}(\Flnk{k+1})$ from Section~\ref{subsec:partial-flags} (see~\eqref{eq:pik-def}) is not a fibration in general: specifically its fibers need not be isomorphic as varieties (or even rationally homotopy equivalent).
That is, there is no analogue of Corollary~\ref{cor:fiber-type} for isomorphism type or rational homotopy type in place of motivic class.
We use the terminology from Section~\ref{subsec:partial-flags}.

\begin{example}%
  \label{ex:fibers-different}
  Consider the case $n=2$, $k=1$, $d_2=2$ and $d_1=3$.
  Given $f \in \Omega^2_{2}(\operatorname{Fl}_{3,2})$ (i.e., a quotient $\OO^3_{\PP^1} \to E_2$ with $\deg E_2 = 2$),
  the fiber $\pi_1^{-1}(f) \subseteq \Omega^2_{2,3}(\operatorname{Fl}_{3,1})$
  is isomorphic to $N_{v_2}(E_2(1))$ (by Lemma~\ref{lem: fibers of Map(Fl(n+1,k))-->Map(Fl(n+1,k+1))}).

  If $E_2 \cong \OO_{\PP^1}(0) \oplus \OO_{\PP^1}(2)$ and $v_2 = (1,0)$ in these coordinates,
  then $N_{v_2}(\OO_{\PP^1}(1) \oplus \OO_{\PP^1}(3))$ may be identified with pairs of polynomials $[f_0(x_0,x_1):f_1(x_0,x_1)]$ with homogeneous degrees $1$ and $3$ respectively, such that $[f_0(1,0):f_1(1,0)] = [1:0]$ and $f_0$ does not divide $f_1$.  Hence the fiber is isomorphic as a variety to $\kk \times (\kk^\ast \times \kk^2)$.

  On the other hand, if $E_2 \cong \OO_{\PP^1}(1) \oplus \OO_{\PP^1}(1)$ and $v_2=(1,0)$, the space $N_{v_2}(\OO_{\PP^1}(2) \oplus \OO_{\PP^1}(2))$ may similarly be identified with coprime pairs of homogeneous degree $2$ polynomials $[f_0(x_0,x_1): f_1(x_0,x_1)]$ such that $[f_0(1,0):f_1(1,0)]=[1:0]$; i.e., with morphisms $f \colon \PP^1 \to \PP^1$ of degree $2$ with $f([1:0])=[1:0]$. By the argument in Example~\ref{ex:homotopy-type} above, this fiber does not have the homotopy type of $S^1$ when $\kk=\CC$, so is certainly not isomorphic as a variety to $\kk^\ast \times \kk^3$ in that case.

  Finally, it is easy to see that there exist points $f \in \Omega^2_{2}(\operatorname{Fl}_{3,2})$ realizing both cases.
\end{example}

\appendix

\section{Detailed discussion of use of AI}%
\label{app:ai-details}

We expand the outline of the discussion of AI started in Section~\ref{subsec: role of AI},
giving further mathematical detail.

The original prompts concerned the following counting analogue (and consequence) of Theorem~\ref{thm: main result}.
\begin{corollary}%
  \label{cor:counting}
  In the set-up of Theorem~\ref{thm: main result}, suppose further that $\kk$ is a finite field, $\kk = \FF_q$.
  Then
  \[
    |\Omega_\beta^2(\Fl)(\kk)| = |{\GL_n(\kk)}|\, q^{D-n^2}.
  \]
\end{corollary}

The sequence of events was then as follows.
\begin{itemize}
  \item First, the AI systems provided complete and readable solutions to Corollary~\ref{cor:counting} in special cases such as $n=2$, $d_1=2$, $d_2=3$ or $n=2$, $0 < d_1 < d_2$ arbitrary.  These cases are relatively concrete; see Appendix~\ref{app:prompts} and Appendix~\ref{app:output} below.
  \item These solutions contained a key step analogous to special cases of Proposition~\ref{prop: key prop giving [N(E)]} when $n=2$. In particular, the key observation was made that the number of choices of $E_1$ does not depend on the splitting type of $E_2$, although not stated in that language.
  \item Human analysis of this proof strategy suggested the statement of Proposition~\ref{prop: key prop giving [N(E)]} in general (or rather, its counting analogue), as well as the approach to proving the main theorem by applying this proposition iteratively to a tower of fibrations $\Omega_\beta^2(\Flnk{k}) \to \Omega_\beta^2(\Flnk{k+1})$, as outlined in Section~\ref{subsec:partial-flags}.
  \item The authors then drafted a new scaffolding including (as a conjecture) a statement analogous to Corollary~\ref{cor:fiber-type}.
  \item Given this, the AI systems were able to complete the proof of Corollary~\ref{cor:counting}.
  \item Subsequently, the systems (given the previous successful solutions) were prompted to follow the same strategy to prove the stronger statement Theorem~\ref{thm: main result}, which they did.
  \item Finally, the model was asked whether an analogue of Proposition~\ref{prop: key prop giving [N(E)]} holds over $\CC$ for homotopy equivalence, rather than equivalence in the Grothendieck ring.
    It successfully produced the counterexample discussed in Section~\ref{subsec:negative-examples} (though we did not follow the model's justification very closely).
\end{itemize}

\section{Example prompts}%
\label{app:prompts}

Overall a significant number of different prompts were presented to the various AI systems as the authors experimented with different approaches.  For brevity we present a couple of representative ones.

Note that the notation and conventions in some of these prompts is \emph{not} compatible with the notation and conventions above.

\subsection{An early scaffold}

The following represents the first few parts of an early ``scaffolding'' of subproblems.
These scaffolds serve (at least) two roles: (i) to explore the limits of the model's capabilities, and (ii) optionally, to provide correct solutions to include in subsequent prompts.

\begin{prompt}
  I will ask you a question in research mathematics, so I will need you to think carefully and rigorously.

  Consider $\mathbb{P}^n$, projective $n$-space over a field.  We will work over the finite field $\mathbb{F}_q$ for now.

  \begin{itemize}
    \item I am interested in studying maps $\pi$ from $\mathbb{P}^1$ to the flag variety $Fl(n+1)$ parametrizing flags in $(n+1)$-space, that send $\infty \in \mathbb{P}^1$ to a fixed flag $f \in Fl(n+1)$.
    \item Now each such map has some numerical invariants, corresponding to the class of $\pi_* \mathbb{P}^1$ in the Chow groups of $Fl(n+1)$.
    \item Precisely: there are $n$ line bundles generating the Picard group of $Fl(n+1)$, and the degrees of the pullback of these line bundles to $\mathbb{P}^1$ are these invariants; I’ll denote them $d_1$ through $d_n$.
  \end{itemize}

  \begin{enumerate}
    \item To make sure you are understanding me, could you tell me which line bundles you would choose on $Fl(n+1)$, and how you would interpret $d_1$ through $d_n$?

    \item Next, please consider the case $n=1$.  How many maps are there (over $\mathbb{F}_q$) when $d_1$ is one of $1$, $2$ or $3$?  Please give a rigorous proof.

    \item When you've done that, please compute the number of points (over $\mathbb{F}_q$) when $n=2$, $d_1=2$ and $d_2=3$.
  \end{enumerate}
\end{prompt}

\subsection{The $n=2$ case in general}%
\label{app:n=2}

\begin{prompt}
    Consider $\mathbb{P}^n$, projective $n$-space over a field.  We will work over the finite field $\mathbb{F}_q$ for now.  I am interested in studying maps $\pi$ from $\mathbb{P}^1$ to the flag variety $Fl(n+1)$ parametrizing flags in $(n+1)$-space, that send $\infty \in \mathbb{P}^1$ to a fixed flag $f \in Fl(n+1)$.  Now each such map has some numerical invariants, corresponding to the class of $\pi_* \mathbb{P}^1$ in the Chow groups of $Fl(n+1)$.

More precisely, let $V$ be an $(n+1)$-dimensional vector space. The flag variety $Fl(n+1)$ parametrizes complete flags $f=(V_1 \subset V_2 \subset \dots \subset V_{n+1}=V)$, where $\dim(V_i)=i$.  The Picard group of $Fl(n+1)$ is isomorphic to $\mathbb{Z}^n$. We choose the generators corresponding to the projections onto the Grassmannians. Let $p_k: Fl(n+1) \to Gr(k, n+1)$ be the projection sending a flag $f$ to the subspace $V_k$.

The Picard group of $Gr(k, n+1)$ is generated by the ample line bundle $\mathcal{O}(1)$, corresponding to the Plücker embedding. We choose the line bundles 
  \[
    L_k = p_k^* \mathcal{O}_{Gr(k, n+1)}(1)
  \]
  for $k=1, \dots, n$ as the generators of $\text{Pic}(Fl(n+1))$.

Given a map $\pi: \mathbb{P}^1 \to Fl(n+1)$, the numerical invariants $d_k$ are defined as:
\[d_k = \deg(\pi^* L_k).\]

Hence, $d_k$ is the degree of the induced map $p_k \circ \pi: \mathbb{P}^1 \to Gr(k, n+1)$. Geometrically, $d_k$ measures the degree of the curve traced out by the $k$-dimensional subspace $V_k(t)$ as $t$ varies in $\mathbb{P}^1$.

\vspace{0.5\baselineskip}
\noindent
\textbf{Research problem:} Find and carefully prove a general formula for the number of maps $\pi$ when $n=2$ and $1 \le d_1 < d_2$ are (otherwise) arbitrary.
\end{prompt}

\subsection{The hinting prompt for general $n$}

\begin{prompt}
\textbf{Background}

Let \(K=\mathbb{F}_q\) be a finite field. We are interested in flags
and partial flags in \(K^{n+1}\). Let:
\begin{itemize}
    \item \(Fl(n+1)\) denote the space
of flags full flags, that is, subspaces
\(0 \le V_1 \le V_2 \le \dots \le V_n \le K^{n+1}\) where \(V_i\) has
dimension \(i\);
    \item \(Fl(n+1; r_1,\dots,r_k)\) denote the space of
partial flags \(0 \le V_1 \le \dots \le V_k \le K^{n+1}\) where
\(\dim(V_i) = r_i\).
\end{itemize}

Note that:
\begin{itemize}
    \item These spaces can be given the structure of algebraic
varieties over \(K\).
    \item In particular, \(Fl(n+1; r_1)\) is just the
Grassmannian variety \(Gr(n+1,r_1)\) and more specifically \(Fl(n+1,1)\)
is just \(\mathbb{P^n}\).
    \item There are \(k\) maps
\(\psi_i \colon Fl(n+1; r_1,\dots,r_k) \to Gr(n+1, r_i)\) given by
mapping the flag \((V_1,\dots,V_k)\) to the subspace \(V_i\).
\end{itemize}
Suppose \(\pi \colon \mathbb{P}^1 \to Gr(n+1,k)\) is a morphism. Recall
that \(\deg(\pi)\) may be defined as follows: take the tautological
vector bundle on \(Gr(n+1,k)\), pull it back to \(\mathbb{P}^1\), and
take the degree of the resulting vector bundle.

Now suppose \(\mathbb{P}^1 \to Fl(n+1; r_1,\dots,r_k)\) is a morphism.
Let \(d_1,\dots,d_k\) be the degrees of \(\pi\), defined by
\(d_i = \deg(\psi_i \circ \pi)\).

Suppose integers \((d_1,d_2,\dots,d_k) \ge 0\) are given. Let
\(M(r_1,\dots,r_k; d_1,\dots,d_k)\) denote the space of morphisms
\(\pi \colon \mathbb{P}^1 \to Fl(n+1; r_1,\dots,r_k)\) having these
degrees.

We are interested in finding the number of points (over \(K\)) of
  \[
    M(r_1,\dots,r_k;d_1,\dots,d_k)
  \]
  in terms of \(q\), \(n\) and the other values.
  \vspace{0.5\baselineskip}

\noindent
\textbf{Question 1}: to warm up, find \(|M(1; d)|\) for all \(d \ge 1\)
and all \(n \ge 1\).
  \vspace{0.5\baselineskip}

\noindent
\textbf{Question 2}: Suppose \(\pi_1 \in M(1; d_1)\) is given, where
\(d_1>0\). For \(n \ge 2\) and any \(d_2>d_1\), determine the number of
points \(\pi \in M(1,2; d_1,d_2)\) whose projection to \(M(1;d_1)\) (by
forgetting the first subspace \(V_2\)) is \(\pi_1\).
  \vspace{0.5\baselineskip}

\noindent
\textbf{Question 3}: if \(0 < d_1 < d_2 < \cdots < d_n\), find
\(|M(1,2,\dots,n; d_1,d_2,\dots,d_n)|\).
\end{prompt}

\section{Example AI outputs}%
\label{app:output}

\newcommand{\F}{\mathbb{F}_q}
\newcommand{\Fs}{\F[s]}
\newcommand{\Qset}[1]{\mathcal{Q}_{#1}}
\newcommand{\Rset}[1]{\mathcal{R}_{#1}}
\newcommand{\degp}{\mathrm{deg}}
\newcommand{\gcdp}{\mathrm{gcd}}
\newcommand{\Npr}{N_{\mathrm{pr}}}
\newcommand{\Ntotal}{N_{\mathrm{total}}}

The following is a solution to the prompt in Appendix~\ref{app:n=2} generated by the system \emph{FullProof} developed by the fourth author (see Section~\ref{subsec: role of AI}).
It has been edited only for formatting compatibility with this document.

\begin{response}

  \noindent
  {\large \bf Introduction and Main Theorem}

  \noindent
The geometric problem of counting maps $\pi: \mathbb{P}^1 \to Fl(3)$ over a finite field $\F$ of a given bi-degree $(d_1, d_2)$ with a fixed base point at infinity can be translated into a purely algebraic problem. A map corresponds to a pair of polynomial triples $(Q, R)$ satisfying specific degree constraints and an incidence relation. This paper provides a rigorous count of such pairs.

\begin{aidefinition}[Polynomial Triple Sets]
Let $\Fs$ be the ring of polynomials in an indeterminate $s$ over a finite field $\F$.
\begin{enumerate}
    \item For an integer $d \ge 1$, we define $\Qset{d}$ to be the set of triples $Q=(q_0, q_1, q_2) \in (\Fs)^3$ where $q_0$ is monic of degree $d$, and $\degp(q_1), \degp(q_2) \le d-1$. We define $\Qset{0} = \{(1,0,0)\}$.
    \item For an integer $d \ge 1$, we define $\Rset{d}$ to be the set of triples $R=(r_0, r_1, r_2) \in (\Fs)^3$ where $r_2$ is monic of degree $d$, and $\degp(r_0), \degp(r_1) \le d-1$.
\end{enumerate}
\end{aidefinition}

\begin{aidefinition}[Intersecting and Primitive Pairs]
A polynomial triple $T=(t_0, t_1, t_2)$ is \textbf{primitive} if its components are coprime, i.e., $\gcdp(t_0, t_1, t_2) = 1$.
A pair of triples $(Q, R)$ with $Q \in \Qset{d_1}$ and $R \in \Rset{d_2}$ is an \textbf{intersecting pair} if their dot product is zero: $Q \cdot R = q_0 r_0 + q_1 r_1 + q_2 r_2 = 0$.
The number we wish to count, $N(d_1, d_2)$, is the number of intersecting pairs $(Q,R)$ where both $Q$ and $R$ are primitive.
\end{aidefinition}

Our main result is the following theorem.

\begin{aitheorem}
For integers $d_1, d_2$ such that $1 \le d_1 < d_2$, the number of primitive intersecting pairs $(Q, R)$ with $Q \in \Qset{d_1}$ and $R \in \Rset{d_2}$ is given by:
\[ N(d_1, d_2) = (q-1)^2(q+1)q^{2d_1+2d_2-3} \]
\end{aitheorem}

\noindent{\large \bf Foundational Counting Lemmas}

\noindent
The proof of the main theorem rests on several key counting results within the ring $\Fs$.

\begin{ailemma}[Sum of M\"obius Function]\label{lem:mobius_sum}
Let $\mu$ be the polynomial M\"obius function defined on monic polynomials in $\Fs$. Let $M_k = \sum_{f \text{ monic}, \degp(f)=k} \mu(f)$. Then $M_0=1$, $M_1=-q$, and $M_k=0$ for all integers $k \ge 2$.
\end{ailemma}
\begin{proof}
Consider the polynomial zeta function $Z(T) = \sum_{f \text{ monic}} T^{\degp(f)}$. The number of monic polynomials of degree $k$ is $q^k$, so $Z(T) = \sum_{k=0}^{\infty} q^k T^k = (1-qT)^{-1}$. The series $\mathcal{M}(T) = \sum_{f \text{ monic}} \mu(f) T^{\degp(f)}$ is the inverse of $Z(T)$ in the ring of formal power series. Thus, $\mathcal{M}(T) = 1-qT$. Comparing coefficients of $\mathcal{M}(T) = \sum_{k=0}^\infty M_k T^k$ with $1-qT$ yields the result.
\end{proof}

\begin{ailemma}[Number of Primitive Triples]\label{lem:n_pr}
Let $\Npr(d)$ be the number of primitive triples in $\Qset{d}$. Then $\Npr(0)=1$, and for $d \ge 1$, $\Npr(d) = q^{3d}-q^{3d-2}$.
\end{ailemma}
\begin{proof}
For $d=0$, $\Qset{0} = \{(1,0,0)\}$, which is primitive, so $\Npr(0)=1$. For $d \ge 1$, we use M\"obius inversion. The number of triples in $\Qset{d}$ where all components are divisible by a monic polynomial $g$ of degree $k \le d$ is $q^{3(d-k)}$.
  \begin{align*}
    \Npr(d) &= \sum_{g \text{ monic}} \mu(g) |\{Q \in \Qset{d} : g|Q\}| \\
    &= \sum_{k=0}^d \Biggl( \sum_{\substack{g \text{ monic}\\\degp g=k}} \mu(g) \biggr) q^{3(d-k)}
    = \sum_{k=0}^d M_k q^{3(d-k)}.
  \end{align*}
Using Lemma~\ref{lem:mobius_sum}, for $d \ge 1$, this sum is $M_0 q^{3d} + M_1 q^{3(d-1)} = (1)q^{3d} + (-q)q^{3(d-1)} = q^{3d}-q^{3d-2}$.
\end{proof}

\begin{aiproposition}[A Key Summatory Function]\label{prop:sigma_d}
Let $\Sigma(d) = \sum_{Q \in \Qset{d}} q^{\degp(\gcdp(Q))}$. Then $\Sigma(0)=1$, and for $d \ge 1$, $\Sigma(d) = q^{3d} + q^{3d-1} - q^{2d-1}$.
\end{aiproposition}
\begin{proof}
For $d=0$, $\Sigma(0)=q^{\degp(1)}=1$. For $d \ge 1$, we group terms by their gcd, $g$. Any $Q \in \Qset{d}$ has a unique representation $Q=gQ'$ where $g$ is monic and $Q'$ is a primitive triple in $\Qset{d-\degp g}$.
\[ \Sigma(d) = \sum_{k=0}^d \sum_{\substack{g \text{ monic}\\\degp g=k}} q^k \Npr(d-k) = \sum_{k=0}^d q^{2k} \Npr(d-k). \]
Let $j=d-k$.
\[\Sigma(d) = \sum_{j=0}^d q^{2(d-j)} \Npr(j) = q^{2d}\Npr(0) + \sum_{j=1}^d q^{2(d-j)}(q^{3j}-q^{3j-2}).\]
\[ \Sigma(d) = q^{2d} + (q^2-1) \sum_{j=1}^d q^{2d-2j+3j-2} = q^{2d} + (q^2-1) q^{2d-2} \sum_{j=1}^d q^j. \]
  \begin{align*} \Sigma(d) &= q^{2d} + (q^2-1) q^{2d-2} \left(\frac{q(q^d-1)}{q-1}\right) \\
    &= q^{2d} + (q+1)q^{2d-1}(q^d-1) = q^{3d} + q^{3d-1} - q^{2d-1}.
  \end{align*}
\end{proof}

\noindent
{\large \bf Counting Total Intersecting Pairs}

\noindent
We first count the total number of intersecting pairs, $\Ntotal(d_a, d_b)$, without the primitivity constraint.

\begin{aiproposition}[Solutions for a Fixed $Q$]\label{prop:solutions_fixed_q}
Let $d_b \ge d_a \ge 1$. For a fixed $Q \in \Qset{d_a}$, the number of $R \in \Rset{d_b}$ such that $Q \cdot R = 0$ is $q^{2d_b - d_a + \delta}$, where $\delta=\degp(\gcdp(Q))$.
\end{aiproposition}
\begin{proof}
Let $R=(r_0, r_1, r_2)$ with $r_0, r_1 \in \{p \in \Fs \mid \degp p < d_b\}$ and $r_2 \in \{p \in \Fs \mid p \text{ is monic, } \degp p = d_b\}$. The condition on $r_2$ means $r_2 = s^{d_b} + \rho_2$ for a unique $\rho_2$ with $\degp \rho_2 < d_b$.
The equation $q_0r_0+q_1r_1+q_2r_2=0$ becomes $q_0r_0+q_1r_1+q_2\rho_2 = -q_2s^{d_b}$.
Let $V = \{p \in \Fs \mid \degp p < d_b\}^3$. This is a vector space of dimension $3d_b$. The number of solutions $(r_0,r_1,\rho_2) \in V$ to this inhomogeneous linear equation is either 0 or $|\ker(\mathcal{L})|$, where $\mathcal{L}:V \to \{p \in \Fs \mid \degp p < d_a+d_b\}$ is $\mathcal{L}(\rho_0,\rho_1,\rho_2) = q_0\rho_0+q_1\rho_1+q_2\rho_2$.
The size of the kernel is the number of solutions to the homogeneous equation $q_0\rho_0+q_1\rho_1+q_2\rho_2=0$. Let $g=\gcdp(Q)$ and $Q'=Q/g$. The equation is equivalent to $Q' \cdot (\rho_0, \rho_1, \rho_2)=0$. Since $Q'$ is primitive, a standard argument shows the solution space in $V$ has dimension $3d_b - (d_a+d_b-\delta) = 2d_b-d_a+\delta$. So $|\ker(\mathcal{L})| = q^{2d_b-d_a+\delta}$.
The solution set is non-empty because the image of $\mathcal{L}$ is precisely the set of polynomials in $\{p \mid \degp p < d_a+d_b\}$ divisible by $g$. The target vector $-q_2s^{d_b}$ is divisible by $g$ and has degree $\degp q_2 + d_b < d_a+d_b$, so it is in the image.
\end{proof}

\begin{aitheorem}[Total Intersecting Pairs]\label{thm:n_total}
For integers $d_b \ge d_a \ge 0$ with $d_b \ge 1$, the total number of intersecting pairs is $\Ntotal(d_a, d_b) = q^{2d_b - d_a} \Sigma(d_a)$.
\end{aitheorem}
\begin{proof}
We sum the result of Proposition~\ref{prop:solutions_fixed_q} over all $Q \in \Qset{d_a}$.
For $d_b \ge d_a \ge 1$:
  \begin{align*}
    \Ntotal(d_a, d_b) &= \sum_{Q \in \Qset{d_a}} q^{2d_b-d_a+\degp(\gcdp(Q))} \\
    &= q^{2d_b-d_a} \sum_{Q \in \Qset{d_a}} q^{\degp(\gcdp(Q))} = q^{2d_b-d_a}\Sigma(d_a).
  \end{align*}
For $d_a=0$, $\Qset{0}=\{(1,0,0)\}$. The equation $Q \cdot R = 0$ becomes $r_0=0$. $R \in \Rset{d_b}$ requires $\degp(r_1) \le d_b-1$ ($q^{d_b}$ choices) and $r_2$ monic of degree $d_b$ ($q^{d_b}$ choices). Thus $\Ntotal(0,d_b)=q^{2d_b}$.
The formula gives $q^{2d_b-0}\Sigma(0) = q^{2d_b}$, so it holds for $d_a=0$ as well.
\end{proof}

\noindent{\large \bf Proof of the Main Theorem}

\noindent
We now have all the components to prove the main theorem by applying M\"obius inversion to the count of total pairs.

\begin{aiproposition}[M\"obius Inversion for Pairs]\label{prop:mobius_inversion}
For $d_1, d_2 \ge 1$, the number of primitive intersecting pairs $N(d_1, d_2)$ is given by
\[ N(d_1, d_2) = \begin{aligned}[t] &\Ntotal(d_1, d_2) - q \Ntotal(d_1-1, d_2) \\&- q \Ntotal(d_1, d_2-1) + q^2 \Ntotal(d_1-1, d_2-1).\end{aligned} \]
\end{aiproposition}
\begin{proof}
By the principle of inclusion-exclusion (or M\"obius inversion on the product of two posets of monic divisors for $Q$ and $R$):
\[ N(d_1, d_2) = \sum_{g,h \text{ monic}} \mu(g)\mu(h) F(g,h), \]
where $F(g,h)$ is the number of intersecting pairs $(Q,R)$ where $g|\gcdp(Q)$ and $h|\gcdp(R)$. A change of variables $Q=gQ'$, $R=hR'$ shows that $F(g,h)=\Ntotal(d_1-\degp g, d_2-\degp h)$. The sum becomes $\sum_{k_1,k_2 \ge 0} M_{k_1}M_{k_2} \Ntotal(d_1-k_1, d_2-k_2)$. Using Lemma~\ref{lem:mobius_sum} truncates the sum to the four terms shown.
\end{proof}

\begin{proof}[\textbf{Proof of the Main Theorem}]
We use Proposition~\ref{prop:mobius_inversion} and the formula for $\Ntotal$ from Theorem~\ref{thm:n_total}. The condition $1 \le d_1 < d_2$ ensures that for all terms $\Ntotal(d_a, d_b)$ appearing, we have $d_b \ge d_a \ge 0$ and $d_b \ge 1$, so Theorem~\ref{thm:n_total} always applies.

First, we group the terms from Proposition~\ref{prop:mobius_inversion}:
\[ N(d_1, d_2) = \begin{aligned}[t] &\left[\Ntotal(d_1, d_2) - q \Ntotal(d_1, d_2-1)\right] \\& - q\left[\Ntotal(d_1-1, d_2) - q \Ntotal(d_1-1, d_2-1)\right].\end{aligned} \]
Let's analyze the first bracket, using $\Ntotal(d_a,d_b)=q^{2d_b-d_a}\Sigma(d_a)$. Since $d_1 < d_2$, we have $d_2-1 \ge d_1 \ge 1$.
  \begin{align*} q^{2d_2-d_1}\Sigma(d_1) - q \cdot q^{2(d_2-1)-d_1}\Sigma(d_1) &= (q^{2d_2-d_1} - q^{2d_2-d_1-1})\Sigma(d_1) \\
    &= (q-1)q^{2d_2-d_1-1}\Sigma(d_1).
  \end{align*}
Similarly, for the second bracket, since $d_1-1 < d_1 < d_2$, we have $d_2 > d_1-1$ and $d_2-1 \ge d_1-1 \ge 0$.
\[ q^{2d_2-(d_1-1)}\Sigma(d_1-1) - q \cdot q^{2(d_2-1)-(d_1-1)}\Sigma(d_1-1) = (q-1)q^{2d_2-d_1}\Sigma(d_1-1). \]
Substituting back, we get a simplified expression for $N(d_1, d_2)$:
\[ N(d_1, d_2) = (q-1)q^{2d_2-d_1-1}\Sigma(d_1) - q\left[(q-1)q^{2d_2-d_1}\Sigma(d_1-1)\right] \]
\[ N(d_1, d_2) = (q-1) \left[ q^{2d_2-d_1-1}\Sigma(d_1) - q^{2d_2-d_1+1}\Sigma(d_1-1) \right]. \]
We now evaluate this final expression.

\textbf{Case 1: $d_1 \ge 2$.} Then $d_1-1 \ge 1$. We use $\Sigma(d_1) = q^{3d_1} + q^{3d_1-1} - q^{2d_1-1}$ and $\Sigma(d_1-1) = q^{3d_1-3} + q^{3d_1-4} - q^{2d_1-3}$.
\begin{align*}
  N(d_1, d_2) &= \begin{aligned} &(q-1) \bigl[ q^{2d_2-d_1-1}(q^{3d_1} + q^{3d_1-1} - q^{2d_1-1}) \\
  &- q^{2d_2-d_1+1}(q^{3d_1-3} + q^{3d_1-4} - q^{2d_1-3}) \bigr] \end{aligned} \\
  &= \begin{aligned} & (q-1) \bigl[ (q^{2d_1+2d_2-1} + q^{2d_1+2d_2-2} - q^{d_1+2d_2-2}) \\& - (q^{2d_1+2d_2-2} + q^{2d_1+2d_2-3} - q^{d_1+2d_2-2}) \bigr] \end{aligned} \\
&= (q-1) \left[ q^{2d_1+2d_2-1} - q^{2d_1+2d_2-3} \right] \\
  &= (q-1)q^{2d_1+2d_2-3}(q^2-1) \\
&= (q-1)^2(q+1)q^{2d_1+2d_2-3}.
\end{align*}

\textbf{Case 2: $d_1 = 1$.} Then $d_1-1 = 0$. We use $\Sigma(1) = q^3+q^2-q$ and $\Sigma(0)=1$.
\begin{align*}
N(1, d_2) &= (q-1) \left[ q^{2d_2-1-1}\Sigma(1) - q^{2d_2-1+1}\Sigma(0) \right] \\
&= (q-1) \left[ q^{2d_2-2}(q^3+q^2-q) - q^{2d_2}(1) \right] \\
&= (q-1) \left[ (q^{2d_2+1} + q^{2d_2} - q^{2d_2-1}) - q^{2d_2} \right] \\
&= (q-1) \left[ q^{2d_2+1} - q^{2d_2-1} \right] = (q-1)q^{2d_2-1}(q^2-1) \\
&= (q-1)^2(q+1)q^{2d_2-1}.
\end{align*}
The main theorem's formula for $d_1=1$ is $(q-1)^2(q+1)q^{2(1)+2d_2-3} = (q-1)^2(q+1)q^{2d_2-1}$, which matches our result.

Both cases yield a result consistent with the formula in the main theorem. This completes the proof.
\end{proof}

\end{response}

\bibliography{BEMSV-revised-biblio}
\bibliographystyle{plain}

\end{document}